\documentclass [12pt,a4paper]{amsart}

\usepackage{amssymb,amsmath,amsthm}

\usepackage{url}

\begin{document}

\newtheorem{theorem}{Theorem}[section]
\newtheorem{lemma}[theorem]{Lemma}
\newtheorem{proposition}[theorem]{Proposition}

\newtheorem{statement}[theorem]{Statement}
\newtheorem{conjecture}[theorem]{Conjecture}
\newtheorem{problem}[theorem]{Problem}
\newtheorem{corollary}[theorem]{Corollary}

\newtheorem*{maintheorem1}{Theorem A}
\newtheorem*{maintheorem2}{Theorem B}

\newtheoremstyle{neosn}{0.5\topsep}{0.5\topsep}{\rm}{}{\bf}{.}{ }{\thmname{#1}\thmnumber{ #2}\thmnote{ {\mdseries#3}}}
\theoremstyle{neosn}
\newtheorem{remark}[theorem]{Remark}
\newtheorem{definition}[theorem]{Definition}
\newtheorem{example}[theorem]{Example}

\numberwithin{equation}{section}

\newcommand{\Aut}{\,\mathrm{Aut}\,}
\newcommand{\Inn}{\,\mathrm{Inn}\,}
\newcommand{\End}{\,\mathrm{End}\,}
\newcommand{\Out}{\,\mathrm{Out}\,}
\newcommand{\Hom}{\,\mathrm{Hom}\,}
\newcommand{\ad}{\,\mathrm{ad}\,}
\newcommand{\SLn}{$\rm sl(n)$}
\newcommand{\calL}{{\mathcal L}}
\renewenvironment{proof}{\noindent \textbf{Proof.}}{$\blacksquare$}
\newcommand{\GL}{\,\mathrm{GL}\,}
\newcommand{\SL}{\,\mathrm{SL}\,}
\newcommand{\PGL}{\,\mathrm{PGL}\,}
\newcommand{\PSL}{\,\mathrm{PSL}\,}
\newcommand{\PSO}{\,\mathrm{PSO}\,}
\newcommand{\SO}{\,\mathrm{SO}\,}
\newcommand{\Sp}{\,\mathrm{Sp}\,}
\newcommand{\PSp}{\,\mathrm{PSp}\,}
\newcommand{\UT}{\,\mathrm{UT}\,}
\newcommand{\Exp}{\,\mathrm{Exp}\,}
\newcommand{\Rad}{\,\mathrm{Rad}\,}
\newcommand{\charr}{\mathrm{char}\,}
\newcommand{\tr}{\mathrm{tr}\,}
\newcommand{\diag}{{\rm diag}}

\font\cyr=wncyr10 scaled \magstep1%

\def\Sha{\text{\cyr Sh}}

\title{Types in torsion free Abelian groups}


\author{Elena Bunina}
\date{}
\address{Department of Mathematics, Bar--Ilan University, 5290002 Ramat Gan, ISRAEL}
\email{helenbunina@gmail.com}

\begin{abstract}
	In this paper we study (logical) types and isotypical equivalence of torsion free Abelian  groups. We describe all possible types of elements and standard $2$-tuples of elements in these groups and classify separable torsion free Abelian groups up to isotypicity. 
\end{abstract}

\maketitle

\section{Introduction}\leavevmode

In this paper we study  torsion free Abelian groups~$A$ from the point of view of (logical) types and isotypical equivalence. 

\begin{definition}
	Let $G$ be a group and $(g_1, \dots , g_n)$ a tuple of its elements. The {\bf type} of this
	tuple in~$G$, denoted $\mathrm{tp}^G(g_1, \dots , g_n)$, is the set of all first order formulas in free
	variables $x_1, \dots , x_n$ in the standard group theory language  which are
	true on $(g_1, \dots , g_n)$ in~$G$ (see~\cite{Mya2} or \cite{Mya3} for details). 
\end{definition}

\medskip

\begin{definition}
	The set of all types of tuples
	of elements of~$G$ is denoted by $\mathrm{tp}(G)$. Following~\cite{Mya7}, we say that two groups
	$G$ and $H$ are {\bf isotypic} if $\mathrm{tp}(G) = \mathrm{tp}(H)$, i.\,e., if any type realized in~$G$ is realized in~$H$, and vice versa.
\end{definition}

\medskip

Isotypic groups appear naturally in
logical (algebraic) geometry over groups which was developed in the works of
B.\,I.\,Plotkin and his co-authors (see~\cite{Mya5,Mya6,Mya7} for details), they play an important
role in this subject. In particular, it turns out that two groups are logically equivalent if and only if they have the same sets of realizable types. So there arise  two fundamental algebraic
questions which are interesting in its own right: what are possible types of
elements in a given group~$G$ and how much of the algebraic structure of~$G$ is
determined by the types of its elements?

Isotypicity property of groups is related to the elementary equivalence property,
though it is stronger. Indeed, two isotypic groups are elementarily equivalent,
but the converse does not hold. For example, if we denote by $F_n$ a free
group of finite rank~$n$, then groups $F_n$ and $F_m$ for $2\leqslant m < n$ are elementarily
equivalent~\cite{Mya22,Mya23}, but not isotypic~\cite{Myas}. Furthermore, Theorem~1 from~\cite{Myas}
shows that two finitely generated isotypic nilpotent groups are isomorphic,
but there are examples, due to Zilber, of two elementarily equivalent non-isomorphic
finitely generated nilpotent of class~$2$ groups~\cite{Mya24}.

Isotypicity is a very strong relation on groups, which quite often implies
their isomorphism. This explains the need of the following definition. 

\begin{definition}
	We say
	that a group $G$ is \emph{defined by its types} if every group isotypic to~$G$ is isomorphic
	to~$G$. 
\end{definition}

It was noticed in~\cite{Myas} that every finitely generated group $G$ which
is defined by its types satisfies a (formally) stronger property. 
Namely, we say
that 

\begin{definition}
	A finitely generated group $G$ is \emph{strongly defined by types} if for any isotypic
	to~$G$ group~$H$ every elementary embedding $G \to H$ is an isomorphism.
\end{definition}

Miasnikov and Romanovsky in the paper~\cite{Myas}  proved that 

1) every virtually polycyclic group is strongly defined by its types;

2) every finitely generated metabelian group is strongly defined by
its types;

3) every finitely generated rigid group is strongly defined by its types.
In particular, every free solvable group of finite rank is strongly defined by its
types.

\smallskip

R.\,Sclinos (unpublished) proved that finitely generated homogeneous groups are defined by types. Moreover, finitely generated co-hopfian and finitely presented hopfian groups are defined by types. Nevertheless, the main problem in the area remains widely open:

\begin{problem}[\cite{Mya6}] 
	Is it true that every finitely generated group is defined by types?
\end{problem}

In the recent paper~\cite{Gvozd_new} Gvozdevsky proved that any field of finite transcendence degree over a prime subfield is defined by types.
Also he gave  several interesting  examples of certain countable isotypic but not isomorphic
structures: totally ordered sets, rings, and groups.

In the papers \cite{Abelian1} and \cite{Abelian2} the author studied periodic Abelian groups, their types and isotypical equivalence. For periodic Abelian groups there exists complete classification of such groups up to isotypicity. The following theorem was proved in \cite{Abelian2}:

\begin{theorem}
	Two periodic Abelian groups $A_1$ and $A_2$ are isotypically equivalent if and only if for every prime~$p$ their $p$-basic subgroups are elementarily equivalent and for any prime~$p$ and natural~$n$ 
	$$
	\dim (p^{n-1} A_1^{p,\infty} )[p] = 	\dim (p^{n-1} A_2^{p,\infty} )[p],
	$$
	where $A^{p,\infty}$ consists of all elements of infinite $p$-height in the $p$-component of~$A$.  
\end{theorem}

Theoretically in this paper our goal could be to show that two torsion free Abelian groups are isotypically equivalent if and only if some their numerical invariants coincide (as in the previous theorem), but in fact it is impossible, because the variety of different logical types of $n$-tuples of elements in these groups is too rich: it contains continuum of parameters and different families of these paramenters can appear in different groups. 

In particular, $1$-types are described  naturally and clearly, $2$-types are described in a more complicated manner, but it is possible to characterize every $2$-type by integer invariants. Even a description of $2$-types shows the compexity of this logical construction of types in torsion free Abelian groups. 

Full description of isotypicity is obtained for the specific class of torsion free Abelian groups: the class of so-called \emph{separable} groups. Of course for other specific classes of torsion free Abelian groups it is theoretically possible to find invariants for isotypical equivalence.

\section{Elementary equivalence of torsion free Abelian groups}\leavevmode

Since elementary equivalence is necessary for isotypicity, we will start with the results of elementary equivalence of Abelian groups.

\begin{definition}
Two groups are called \emph{elementarily equivalent} if their first order theories coincide.
\end{definition}

Elementary equivalent Abelian groups were completely described in 1955 by Wanda Szmielew in~\cite{a1} (see also Eclof and Fisher,~\cite{a2}).

To formulate her theorem we need to introduce a set of special invariants of Abelian groups.

Let $A$ be an Abelian group, $p$ a prime number, $A[p]$ be the subgroup of~$A$, containing all elements of~$A$ of the orders $p$ or~$1$ (it is so-called $p$-\emph{socle} of the group~$A$),
$kA$ be the subgroup of~$A$, containing all elements of the form $ka$, $a\in A$.

\medskip

{\bf The first invariant} is
$$
D(p;A):= \lim\limits_{n\to \infty} \dim ((p^nA)[p])\text{ for every prime }p.
$$

\medskip

{\bf The second invariant} is
$$
Tf (p;A):= \lim\limits_{n\to \infty} \dim (p^nA/p^{n+1}A).
$$

\medskip

{\bf The third invariant } is
$$
U(p,n-1;A):= \dim ((p^{n-1}A)[p] / (p^nA)[p]).
$$

\medskip

{\bf The last invariant} is $\Exp(A)$ which is an \emph{exponent} of~$A$ (the smallest natural number  $n$ such that $\forall a\in A\, na=0$).

\begin{theorem}[Szmielew theorem on elementary classification of Abelian groups,~\cite{a1}]\label{EE-Abelian}
Two Abelian groups  $A_1$ and $A_2$ are elementarily equivalent if and only if their elementary invariants  $D(p;\cdot)$, $Tf (p;\cdot)$, $U(p,n-1;\cdot)$ and $\Exp(\cdot)$
pairwise coincide for all natural $n$ and prime~$p$ \emph{(}more precisely, they are either finite and coincide or simultaneously are equal to infinity\emph{)}.
\end{theorem}

It is clear that if some group $B$ is elementarily equivalent to the given torsion free Abelian group~$A$, then $B$ is also a torsion free Abelian group.

If a group $A$ is torsion free, then $D(p;A)=0$ for all $p$, $U(p,n-1,A)=0$ for all $p$ and~$n$, $\Exp (A)=\infty$.

So the only meaningful invariants are $Tf(p;A)$ for different prime~$p$. Therefore we have the following 

\begin{corollary}
	Two torsion free Abelian groups  $A_1$ and $A_2$ are elementarily equivalent if and only if  $Tf (p;\cdot)$
	pairwise coincide for all  prime~$p$.
\end{corollary}

\section{$1$-Types of elements in  torsion free Abelian groups and $1$-isotypicity}

\subsection{Formulas in Abelian groups}\leavevmode

We will use the following elimination of quantifiers in Abelian groups:

\begin{proposition}[\cite{Mya2}, Lemma A.2.1, Pr\"ufer and \cite{Baur}, Baur]\label{LemmaA21}
	Every   formula in the language of Abelian groups  is equivalent   to a boolean combination of $\forall \exists$-sentences and a finite conjunction of formulas of the forms $\alpha_1 x_1+\dots +\alpha_k x_k=0$ and and  $p^n |  \alpha_1 x_1+\dots +\alpha_k x_k$, where $\alpha_1,\dots, \alpha_k\in \mathbb Z$,  $p$ is primes and $n$ is a natural number. 
\end{proposition}

According to this statement we have the following situation: 
the type of any tuple $(a_1,\dots, a_m)\in A$ (here $A$ is any Abelian group, not necessarily periodic) contains two parts: 

(1) Elementary theory of~$A$;

(2) All formulas
$$
\alpha_1 a_1+\dots +\alpha_m a_m=0,\quad p^n | \alpha_1 a_1+\dots +\alpha_m a_m,
$$
with $\alpha_1,\dots, \alpha_m\in \mathbb Z$, $p$ prime and $n$ natural, that hold in~$A$.

In other words, for any torsion free Abelian $p$-group $A$ possible types of elements describe the elementary theory of~$A$, and orders (which are zero or $\infty$) and $p$-heights of linear combinations of the corresponding elements for different prime~$p$.

\subsection{$1$-Types of torsion free Abelian groups}\leavevmode

$1$-Type of an element $a\in A$   consists of all formulas with one free variable, that are true on this element~$a$. As we saw above, for any $a\in A$ this type consists of the whole elementary theory of~$A$, all statements about orders and $p$-heights of all elements $\alpha \cdot a$, $\alpha \in \mathbb Z$. 

For a torsion-free group all orders of all these elements (except~$0$) are infinite, so we are interested only in $p$-heights, $p$ is an arbitrary prime number.

Note that for torsion free Abelian groups if $h_p(a)=k$, then $h_p (p^\ell a)=k+\ell$, so if we know  $p$-heights of~$a$ for all prime~$p$, then we know  $p$-heights of all elements $\alpha \cdot a$ for all prime~$p$.

It means that it is very useful to introduce a notion of a \emph{characteristic} or a \emph{height sequence}  of~$a$ (see~\cite{Fuks}, Chapter XIII):
$$
\chi (a)=(h_{p_1} (a),\dots, h_{p_n}(a),\dots ),
$$
where $p_1,\dots, p_n,\dots $ is a sequence of all prime numbers.

\smallskip

\begin{definition}
	Two groups are called $n$-\emph{isotypically equivalent}, if they realize the same sets of $m$-types, $m\leqslant n$. 
\end{definition}

The result about $1$-isotypicity of torsion free Abelian groups is now evident:

\begin{proposition}\label{1-types-1}
	Two torsion free Abelian groups $A_1$ and $A_2$ are $1$-isotypically equivalent if and only  if they are elementarily equivalent and for any element $a\in A_1$ there exists an element $b\in A_2$ with the same characteristic, and vice versa.
\end{proposition}

\subsection{Characteristics and h-types}\leavevmode

Here we wil put some theory from~\cite{Fuks}, which will be  useful for us.

Let us set 
$$
(k_1,\dots, k_n,\dots ) \leqslant (\ell_1,\dots, \ell_n,\dots),
$$
if $k_n \leqslant \ell_n$ for all~$n$. Then the set of all characteristics can be converted into the Dedekind structure  according to component-wise operations:
$$
(k_1,\dots, k_n,\dots ) \cap (\ell_1,\dots, \ell_n,\dots)= (\min (k_1,\ell_1),\dots, \min (k_n,\ell_n),\dots)
$$
and 
$$
(k_1,\dots, k_n,\dots ) \cup (\ell_1,\dots, \ell_n,\dots)= (\max (k_1,\ell_1),\dots, \max (k_n,\ell_n),\dots).
$$
In this structure $(0,0,\dots, 0,\dots)$ is the smallest element and $(\infty,\dots, \infty,\dots)$ is the greatest element. 

For a torsion free group $A$ we clearly have:

(a) $\chi_C(c) \leqslant \chi_A(c)$ for all elements $c$ of a subgroup $C$ of~$A$;

(b) $\chi (b+c)\geqslant \chi (b)\cap \chi (c)$ for all $b,c\in A$;

(c) if $A=B\oplus C$ and $b\in B$, $c\in C$, then $\chi (b+c) =\chi (b) \cap \chi (c)$.

Two characteristics $(k_1,\dots, k_n\dots)$ and $(\ell_1,\dots ,\ell_n,\dots)$ will be considered \emph{equivalent}, if $k_n\ne \ell_n$  only for a finite number of poisitons and only if $k_n $ and $\ell_n$ are finite. 

In \cite{Fuks} an equivalence class of characteristics is called a \emph{type}, but we use this words for a logical type, therefore we will call it an \emph{h-type}. If $\chi(a)$ belongs to a type~$\mathbf t$, then we will say that \emph{an element $a$ has an h-type~$\mathbf t$} and write $\mathbf t(a)=\mathbf t$. 

An h-type $\mathbf t$ is represented by any characteristic belonging to this h-type. The set of all h-types also is a distributive structure. For h-types $\mathbf t$ and $\mathbf s$ we can write $\mathbf t \geqslant \mathbf s$, if there exist characteristics $(k_1,\dots, k_n,\dots)$ and $(\ell_1,\dots, \ell_n,\dots)$ belonging to $\mathbf t$ and $\mathbf s$, respectively, such that    $(k_1,\dots, k_n,\dots)\geqslant (\ell_1,\dots, \ell_n,\dots)$.

If $a$ and $b$ are dependent in~$A$, then $\mathbf t(a)=\mathbf t(b)$. Also $\mathbf t(a+b)\geqslant \mathbf t(a)\cap \mathbf t(b)$ for all $a,b\in A$.

To an arbitrary h-type $\mathbf t$ we can correspond two fully characteristics subgroups of~$A$:
$$
A(\mathbf t):=\{ a\in A\mid \mathbf t(a)\geqslant \mathbf t\}\text{ and }A^*(\mathbf t):=\{ a\in A\mid \mathbf t(a)> \mathbf t\}.
$$

\begin{remark}
	If in a torsion free Abelian group $A$ there exsists an element $a$ of characteristic~$\chi\in \mathbf t$, then for any other $\chi'\in \mathbf t$ there exists $b\in A$ with $\chi(b)=\chi'$.
	
	\medskip
	
	Actually, if $\chi=(k_1,\dots, k_n,\dots),\chi'=(\ell_1,\dots, \ell_n,\dots)\in \mathbf t$, then there exists only a finite number of prime numbers $p_{i_1},\dots, p_{i_m}$ such that $k_{i_j}\ne \ell_{i_j}$ and all this elements are finite. for every $j=1,\dots, m$ let us multiply $a$ by $p_{i_j}^s$, if $\ell_{i_j}=k_{i_j}+s$, or divide $a$ by $p_{i_j}^s$, if $\ell_{i_j}=k_{i_j}-s$. After all these multiplications and divisions we will get the required~$b$.
\end{remark}

Therefore we can reformulate Proposition~\ref{1-types-1}, using h-types:

\begin{proposition}\label{1-types-2}
	Two torsion free Abelian groups $A_1$ and $A_2$ are $1$-isotypic if and only if they are elementarily equivalent and the sets of their h-types coincide.
\end{proposition}

\section{$2$-Types of elements in  torsion free Abelian groups}

\subsection{Independence and heights-independence}\leavevmode

Studying types of elements we are interested to include our $n$-tuples of elements into $m$-tuples of elements so that:

(1) the first tuple consists of linear combinations of the elements of the second tuple;

 (2) the elements of the second tuple are \emph{independent} in all senses that we can find for them.
 
 For example, if we speak about $n$ elements of the group $\mathbb Z \oplus \dots \oplus \mathbb Z$, we can find $m$ linearly independent elements of this group, generating its direct summand and such that our initial elements are their linear combinations. All these elements and their linear combinations with greatest common divisor of all coefficients equal to one will have characteristic $(0,\dots, 0,\dots)$. It is the best situations, and therefore we want to come to it as close as possible. 
 
 Linear independence can be reached at any stage:
 
 \begin{lemma}\label{independence}
 	If $a_1,\dots, a_n\in A$ and $A$ is a torsion free Abelian group, then there exist linearly independent $b_1,\dots, b_m$, $m\leqslant n$, such that $a_1,\dots, a_n$ are linear combinations of $b_1,\dots, b_m$. If $a_1,\dots, a_n$ where linearly dependent, then $m< n$.
 \end{lemma}
 
 \begin{proof}
 	Let us prove it by induction by~$n$.
 	
 	For $n=1$ the statement is evident.
 	
 	For $n=2$ we have $\alpha a_1=\beta a_2$ and according to infinite orders of all elements we can assume that $\alpha$ and $\beta$ are coprime. If $p^k | \alpha$, then $h_p(a_2)\geqslant k$ and we can find $b\in A$ such that $a_2=p^k\cdot b$. Therefore now we have the relation $\frac{\alpha}{p^k} \cdot a_1=\beta b$. Therefore by a finite number of steps we can come to $c\in A$ such that $a_1=\beta \cdot c$, $a_2=\alpha \cdot c$, what was required.
 	
 	Let us now assume that the statement holds for~$n$ and prove it for $n+1$.
 	
 	If we have some relation $\alpha_1 a_1+\dots + \alpha_na_n+\alpha_{n+1}a_{n+1}=0$, then we can also asume that $gcd (\alpha_1,\dots, \alpha_{n+1})=1$. Let us denote
 	$$
 	\alpha_n a_n+\alpha_{n+1} a_{n+1}= \alpha (\alpha_n' a_n + \alpha_{n+1}' a_{n+1}) = \alpha b,\text{ where } \alpha =\gcd (\alpha_n, \alpha_{n+1}).
 	$$
 	Then we have
 	$$
 	\alpha_1 a_1+\dots+ \alpha_{n-1} a_{n-1} +\alpha b=0,\text{ were }\gcd (\alpha_1,\dots, \alpha_{n-1},\alpha)=1,
 	$$ and by induction assumption there  exist linearly independent $b_1,\dots, b_m$, $m< n$, such that $a_1,\dots, a_{n-1}, b$ are expressed trough them. 
 	
 	Now $\alpha_n' a_n + \alpha_{n+1}' a_{n+1}=b$ and there exist $\beta,\gamma\in \mathbb Z$ such that 
 	$\beta \alpha_n' - \gamma \alpha_{n+1}'=1$. If $b'= \gamma a_n + \beta a_{n+1}$, then 
 	$$
 	\gamma b-\alpha_n' b'= (\gamma \alpha_n' - \gamma  \alpha_n') a_n + (\beta \alpha_n' - \gamma \alpha_{n+1}') a_{n+1}=a_{n+1}
 	$$
 	and 
 	$$
 	\beta b- \alpha_{n+1}'b'= a_n,
 	$$
 	so  $a_n$ and $a_{n+1}$ are expressed trough $b_1,\dots, b_m, b'$ and our initial linear dependence become a dependence of $\leqslant n$ elements, which is considered above by induction.
 \end{proof}
 
 \medskip
 
 Therefore we always can assume that in our tuple (or in any derivative tuple) all elements are independent. 

Let us now consider situations with $p$-heights of linear combinations of elements for different prime $p$. 

The following lemma is clear.

\begin{lemma}
	If $a, b\in A$ and $h_p(a)< h_p(b)$, then always $h_p(a+b)=h_p(a)$.
\end{lemma} 

In direct sums if elements $a$ and $b$ belong to different direct summands and $h_p(a)=h_p(b)$, then always $h_p(a+b)=h_p(a)$. Therefore all non-standard situations are situations when $h_p(a)=h_p(b)$
and $h_p(a+b)> h_p(a)$.

So in an ideal situation we would try to find instead of our tuple $a_1,\dots, a_n$ an independent tuple $b_1,\dots, b_m$ such that all $a_1,\dots, a_n$ are expressed trough $b_1,\dots, b_m$ and for all prime~$p$ for every linear combination $\alpha_1 b_1+\dots +\alpha_m b_m$ we have
$$
h_p(\alpha_1 b_1+\dots +\alpha_m b_m)= \min \{ h_p (\alpha_i b_i) \mid i=1,\dots, m)\}.
$$
Therefore it is natural to introduce the following definition:

\smallskip

\begin{definition}
	Elements $a_1,\dots, a_n\in A$ are called \emph{$p$-height-independent}, if for all $\alpha_1,\dots, \alpha_n\in \mathbb Z$
	$$
	h_p(\alpha_1 a_1+\dots +\alpha_n a_n)= \min \{ h_p (\alpha_i a_i) \mid i=1,\dots, n)\}.
	$$
	If elements $a_1,\dots, a_n$ are $p$-height-independent for all prime~$p$, then they are called \emph{height-independent}.
\end{definition}

\smallskip

\subsection{$p$-Height-independence of pairs of elements}\leavevmode 

Let us show the most illustrative case, when we have only two elements $a$ and $b$ and want to consider all possible combinations of their $p$-height-dependence for a fixed prime~$p$.

Assume that $a,b\in A$, $h_p(a)=k$, $h_p(b)=\ell$, $\ell \geqslant k$,  and $a$ and $b$ are not $p$-height-independent.  We need to check only linear combinations of the form
$$
\alpha p^{\ell-k} a +\beta b,\quad \gcd(\alpha,\beta)=1,\quad p\not| \alpha,\ p\not| \beta,\ \alpha >0.
$$
Let us take a linear combination $c=\alpha p^{\ell-k} a+\beta b$ satisfying this condition such that $h_p(c)=r$, $r> \ell$. Since $h_p(p^{\ell-k+1}a)= h_p (pb)=\ell+1> \ell$, then $h_p((\alpha + \gamma_1 p)p^{\ell-k}a + (\beta +\gamma_2)b)> \ell$ for all $\gamma_1,\gamma_2\in \mathbb Z$. Therefore if there exists any dependence, then there exists a dependence with $0< \alpha, \beta < p$. 

So now for cheking, if there exists a  $p$-height-independence of a pair $a,b$ we will check only all linear combinations
 $$
 \alpha p^{\ell-k} a +\beta b,\quad 0< \alpha,\beta < p, \quad  \gcd(\alpha,\beta)=1.
 $$
 
 \begin{lemma}\label{uniqueness}
 	For $a,b\in A$, $h_p(a)=k$, $h_p(b)=\ell$, $\ell \geqslant k$, two different $p$-height-dependencies with different $0< \alpha,\beta < p$, $\gcd(\alpha,\beta)=1$, and $0< \gamma,\delta < p$, $\gcd(\gamma,\delta)=1$, are impossible. 
 \end{lemma}
 
 \begin{proof}
 	Let us assume that there exist two different dependencies with $\alpha,\beta$ and $\gamma,\delta$ satisfying the assertion of the lemma. Let us choose from all dependences with coprime coefficients from the interval $(0;p)$ two of them with the smallest possible~$\alpha$ and $\gamma$, and among the same~$\alpha$ or $\gamma$ --- with the smallest possible~$\beta$ or~$\delta$. Then if $\alpha \ne \gamma$, then for some $q$ the element
 	$$
 	(\gamma p^{\ell-k} a+\delta b) -q(\alpha p^{\ell-k} a+ \beta b) =(\gamma -q \alpha )p^{\ell-k} a+ (\delta-q\beta )b
 	$$
 	or maybe  the element
 	$$
 	(\gamma -q \alpha )p^{\ell-k} a+ (\delta-q\beta )b + pb
 	$$
 	has also a $p$-height $>\ell$ and the first coefficient smaller than~$\alpha$. Contradiction.
 	
 	If $\alpha =\gamma$, then $(\beta - \delta)b$ has a $p$-height greater than~$b$, which is impossible.
 	
 	So the lemma is proved. 
 \end{proof}
 
\begin{lemma}\label{inf_height}
	If for $a,b\in A$, $h_p(a)=k$, $h_p(b)=\ell$, $\ell \geqslant k$, for some $\alpha,\beta$ with $\gcd(\alpha,\beta)=1$, $p\not| \alpha,\beta$, we have $h_p(\alpha p^{\ell-k} a+\beta b) =\infty$, then there exist independent and $p$-height-independent $c,d\in A$, $h_p(c)=\infty$, such that $a,b$ are expressed trough $c,d$. 
\end{lemma}

\begin{proof}
	Let $c$ be such that $p^{\ell-k} c= \alpha p^{\ell-k} a+\beta b$. Then of course $h_p(c)=\infty$. Let also $\gamma,\delta\in \mathbb Z $ be such that $\gamma \alpha p^{\ell-k} - \delta \beta=1$. Let $d=\delta a + \gamma b$. Then 
	$$
	\delta p^{\ell-k} c- \alpha p^{\ell-k} d=\alpha \delta p^{\ell-k} a + \beta \delta b - \alpha \delta p^{\ell-k} a- \alpha \gamma p^{\ell-k} b=-b,
	$$ 
	so $b$ is expressed trough $c$ and~$d$. 
	
	Similarly,
	$$
	\gamma p^{\ell -k} c - \beta d=\alpha \gamma p^{\ell-k} a +\gamma \beta b- \beta \delta a - \beta \gamma b= a,
	$$
	so $a$ is expressed trough $c$ and $d$.
	
	If $h_p(d)=\infty$, then a $p$-height of any linear combination of $c$ and $d$ is~$\infty$, contradiction. So $d$ has a finite $p$-height and therefore $c$ and $d$ are independent and $p$-height-independent.
\end{proof}

So now we see that we can have different situations:

(1) $a$ and $b$ are $p$-height-independent: this situation is good and clear.

(2) There is some dependency of $a$ and $b$ (both of them have finite $p$-heigth) which has an infinite $p$-height. In this case $a$ and $b$ are expressed trough two independent and $p$-height-independent elements, one of them has an infinite height, another one has a finite  height. This situation is also good and clear.

(3) There is some dependency $\alpha p^{\ell-k} a + \beta b$ of $a$ and $b$ of a $p$-height $r> \ell$, where $0< \alpha,\beta < p$, $\gcd(\alpha,\beta)=1$, and then all other dependencies can be only of the form
$$
(\alpha p^{\ell -k} + \gamma p^{s-k}) a+ (\beta + \delta p^{s-\ell}) b,\text{ where } s\geqslant r.
$$
This is the most unclear situation, so we need to show, that there exist examples with almost arbitrary such dependencies.

\begin{example}\label{ex-p-adic}
	Let us consider the group $J_p$ of $p$-adic numbers. These numbers can be described  as infinite sequences   $x=(x_1, x_2,\dots, x_n,\dots)$ of residues $x_n$ modulo $p^n$ with the relations $x_n\equiv x_{n+1} \mod p^n$, $n\in \mathbb N$,  and component-wise addition modulo~$p^n$. Such an element $x$ has a $p$-height $\ell$ if and only if for all $i\in \{1,\dots, n\}$ $p^\ell | a_i$, but $a_{\ell+1}\not\equiv 0\mod   p^{\ell+1}$. 
	
	\smallskip

	Without loss of generality we can consider two elements $a$ and $b$ of $p$-heights~$0$.
	
	\smallskip 
	
	So let $a=(a_1, a_2, \dots, a_n,\dots)$  and $b=(b_1,\dots, b_n,\dots)$ be two elements of~$J_p$ of the $p$-height~$0$.
	
	\smallskip
	
	Let also $0< k_1< k_2 < \dots < k_n < \dots$ be an infinite    increasing sequence of natural numbers,   $(\alpha_1,\beta_1)$, \dots, $(\alpha_n,\beta_n)$, \dots be an infinite sequence of pairs of coprime natural numbers between $0$ and~$p$. Assume also that for any $i=1,\dots, n,\dots$ the pair $(\alpha_{i+1},\beta_{i+1})$ is not proportional to the pair
	$$
	(\alpha_1+\alpha_2 p^{k_1}+\dots + \alpha_i p^{k_{i-1}}, \beta_1+\beta_2 p^{k_1} +\dots + \beta_i p^{k_{i-1}}).
	$$
	
	\smallskip
	
	Our goal is two find $a$ and $b$ such that for all $i\in \mathbb N$ an element
	$$
	c_i=(\alpha_1 a + \beta_1 b) +  p^{k_1}(\alpha_2 a + \beta_2 b) + p^{k_2} (\alpha_3 a+ \beta_3 b) + \dots + 
	p^{k_{i-1}}(\alpha_i a + \beta_i b)
	$$
	has a $p$-height $k_i$. 
	
	Let us do it sequentially, starting with $i=1$.
	
	\smallskip
	
	If $i=1$, then we have $0< \alpha_1,\beta_1< p$, $k_1 > 0$, and need to find  $0\leqslant a_1,b_1< p$ such that $h_p(\alpha_1 a + \beta_1 b) =k_1$. If we take $a=(\beta_1,\dots, \beta_1, a_{k_1+1},\dots)$ and $b=(-\alpha_1,\dots, -\alpha_1, b_{k_1+1},\dots)$, then  $\alpha_1 a + \beta_1 b=(0,\dots, 0, c_{k_1+1},\dots)$ has a $p$-height at least~$k_1$.  To have it precisely~$k_1$ we need to define $a_{k_1+1}\equiv \beta_1 \mod p^{k_1}$ and $b_{k_1+1}\equiv -\alpha_1 \mod p^{k_1}$ so that $\alpha_1 a_{k_1+1} +\beta_1 b_{k_1+1}\ne 0$.  So $a_{k_1+1}= \beta_1 + p^{k_1} \gamma$ and $b_{k_1+1}=-\alpha_1 +p^{k_1} \delta$ such that $\alpha_1 \gamma +\beta_1 \delta \ne 0 \mod p$. It just means that the pair $\gamma,\delta$ is not proportional to the pair $\beta_1, -\alpha_1$ and for the next step it means that the pairs $\alpha_1,\beta_1$ and $\alpha_2,\beta_2$ are not proportional, but this is our assumption. 
	
	\smallskip
	
	After that we can take new 
	$$
	a':= \frac{1}{p^{k_1}}(a - (\beta_1,\beta_1,\dots, \beta_1,\dots))=\frac{1}{p^{k_1}}(\beta_1-\beta_1,\dots, \beta_1-\beta_1, a_{k_1+1}-\beta_1,\dots )=(\gamma, \dots )
	$$
	and the same $b'=(\delta,\dots)$ and apply the previous step to them with coefficients $(\alpha_2,\beta_2)$ and $k_2-k_1$ as the next $p$-height. 
	
	So it is clear that we can for any initial data find two elements $a$ and $b$ in $J_p$ with the corresponding $p$-height-dependencies. 
	
	Therefore we see that the described situation with an infinite number of different depedencies can be realized in the group~$J_p$.
	
	If we have only finite number $n$ of different dependencies (or a $p$-height-independence), we can realize this case inside $J_p\oplus J_p$, where  these $n$ dependencies are realized as above by the sequences $((a_1,\dots, a_{k_n},\dots), (0,\dots, 0,\dots))$ and  $((b_1,\dots, b_{k_n},\dots), (0,\dots, 0,\dots))$ and then we add the tail
	$$
	((a_1,\dots, a_{k_n}, a_{k_n}+p^{k_n}, a_{k_n}+p^{k_n}, \dots,a_{k_n}+p^{k_n}, \dots), (0,\dots, 0,\dots))
	$$
	and 
	$$
	((b_1,\dots, b_{k_n}, b_{k_n}, b_{k_n}, \dots), (0,\dots, 0, p^{k_n}, p^{k_n}\dots, p^{k_n},\dots )).
	$$
\end{example}	

\medskip

Therefore we see that there are some natural boundaries for $p$-height-dependencies of two elements, but taking into account these boundaries we can realizes all admissible types of $p$-height-dependencies of two elements (without elements of infinite $p$-heights) even inside the group $J_p$ or $J_p\oplus J_p$. According to the previous lemmas in the group $J_p\oplus J_p\oplus \mathbb Q \oplus \mathbb Q$ we can realize absolutely all admissible $p$-height-dependencis of two elements.

\subsection{$2$-types in torsion free Abelian groups}\leavevmode

From all previous lemmas and the previous example we can formulate the list of all ``standard'' $2$-types of elements in torsion free Abelian groups.

\begin{theorem}[$2$-types in torsion free Abelian groups]
	Any two elements $x$ and $y$ in a torsion free Abelian group~$A$ are linear combinations either  of one element~$a\in A$ of characteristic~$\chi$, or of two independent elements $a,b\in A$ such that for every prime~$p$ one of the following conditions holds:
	
	\emph{(1)} $h_p(a)=\infty$ or $h_p(b)=\infty$ and $a$ and $b$ are $p$-height-independent.
	
	\emph{(2)} $h_p(a)=k$, $h_p(b)=\ell$ and $a$ and $b$ are $p$-height-independent.
	
	\emph{(3)} $h_p(a)=k$, $h_p(b)=\ell$, $h_p(\alpha a+\beta b)=\infty$ for some coprime $(\alpha, \beta)$ and in this case $a$ and $b$ are expressed as linear combinations trough $p$-height-independent $c$ of infinite $p$-height and $d$ of a finite $p$-height.
	
	\emph{(4)} $h_p(a)=k$, $h_p(b)=\ell$, $\ell \geqslant k$, and there exists a (finite or infinite) sequence of natural numbers $\ell < t_1 < t_2 < \dots < t_n < \dots$ and pairs $(\alpha_1,\beta_1)$, \dots, $(\alpha_n,\beta_n)$, \dots of coprime numbers between $0$ and~$p$ such that for every $i=1,\dots,n,\dots$ the pair  $(\alpha_{i+1},\beta_{i+1})$ is not proportional to 
	$$
	(\alpha_1+\alpha_2 p^{t_1}+\dots + \alpha_i p^{t_{i-1}}, \beta_1+\beta_2 p^{t_1} +\dots + \beta_i p^{t_{i-1}})
	$$
	 and for all $i=1,\dots,n,\dots$ the element
	$$
	c_i=(\alpha_1+\alpha_2 p^{t_1}+\dots + \alpha_i p^{t_{i-1}}) p^{\ell-k} a+(\beta_1+\beta_2 p^{t_1} +\dots + \beta_i p^{t_{i-1}}) b
	$$
	has a $p$-height~$t_i$.
	
	All these $2$-types can be realized in some Abelian group.
\end{theorem}

\begin{proof}
	Let us take two elements $x,y\in A$ and then by Lemma~\ref{independence}  there exist $a\in A$ or independent $a,b\in A$ such that $x,y$ are expressed trough $a$ or $a,b$.
	
	If $x,y$ are expressed trough~$a$, then we know that $a$ is defined by its characteristic~$\chi$.
	
	Now let us assume that $x,y$ are expressed trough two independent~$a,b$.
	
	Let us fix some prime~$p$ and consider $p$-heights of $a,b$ and all their linear combinations.
	
	If $h_p(a)=h_p(b)=\infty$, then $p$-heights of all their linear combinations are infinite, so by definition $a$ and $b$ in this case are $p$-height-independent. Therefore we  can assume that at least one  of $a$ and $b$ has a finite $p$-height.
	
	If $h_p(a)=\infty$, $h_p(b)=\ell$, then for all $\alpha \in \mathbb Z$ $h_p(\alpha a)=\infty$, for all $\beta\in \mathbb Z$ coprime with~$p$ $h_p(\beta p^k b)=\ell+k$. So $h_p(\alpha a+\beta p^k b) = \ell+k$ and therefore $a$ and $b$ are $p$-height-independent.
	
	Consequently, we have considered all cases in (1) and can assume that both $a$ and $b$ have finite $p$-heights.
	
	Let us now assume that $h_p(a)=k$, $h_p(b)=\ell$, $\ell \geqslant k$ and $a$, $b$ are not $p$-height-independent. Therefore there exists some dependency $\alpha a+\beta b$ and we of course can assume that $\alpha$ and $\beta$ are comprime. If there exists such a linear combination with an  infinite $p$-height, then by Lemma~\ref{inf_height} there exist independent and $p$-height-independent $c$ and $d$, $h_p(c)=\infty$, $h_p(d)< \infty$, such that $a$ and $b$ are linear combinations of $c$ and~$d$. It is precisely the case~(3).
	
	Now we have only situation of $p$-height-dependent $a$ and $b$, where all dependencies have finite $p$-heights.
	
	We always can assume that our dependencies have the form
	$$
	\alpha p^{\ell-k} a+ \beta b,\qquad p\not| \alpha, p\not| \beta,\quad \gcd(\alpha,\beta)=1,
	$$
	where $h_p(\alpha p^{\ell-k} a+ \beta b) > \ell$. Taking $\alpha$ and $\beta$ modulo~$p$ means subtracing from this element some linear combination of $p^{\ell-k} a$ and $b$ with coefficients divided by~$p$. These elements have $p$-heights $> \ell$, therefore we have a linear combination $\alpha_1 p^{\ell-k} a+ \beta_1 b$, $0< \alpha_1,\beta_1 < p$, $\gcd(\alpha_1,\beta_1)=1$ of a $p$-height $t_1> \ell$. 
	
	Let us take this dependency $\alpha_1 p^{\ell-k} a+ \beta_1 b$ with $0< \alpha_1,\beta_1 < p$, $\gcd(\alpha_1,\beta_1)=1$ of a $p$-height~$t_1> \ell$. If all other dependencies are obtained from this dependency by multiplication by some integer number, then we can stop. Let us have some another dependency $\gamma p^{\ell-k} a+ \delta b$. By Lemma~\ref{uniqueness} $\gamma \equiv \alpha_1 \mod p$, $\delta \equiv \beta_1\mod p$,
	so
	$$
	 \gamma p^{\ell-k} a+ \delta b=(\alpha_1+ \alpha_2 p^s)p^{\ell-k} a +(\beta_1+ \beta_2 p^r) b=(\alpha_1 p^{\ell-k} a+ \beta_1 b)+(\alpha_2 p^s\cdot p^{\ell-k} a +\beta_2 p^r  b).
	 $$
	 Since this last sum is a new dependency and $h_p(\alpha_1 p^{\ell-k} a+ \beta_1 b)=t_1$, then the $p$-height of this sum is greater than~$t_1$ and $h_p(\alpha_2 p^s\cdot p^{\ell-k} a +\beta_2 p^r  b)=t_1$. 
	 
	 If $s< r$ or $r< s$, then there is dependency $(\alpha_1 p^{\ell-k} a+ \beta_1 b)+\alpha_2 p^s\cdot p^{\ell-k} a$ or $(\alpha_1 p^{\ell-k} a+ \beta_1 b)+\beta_2 p^r  b$ and therefore there exists a dependency $(\alpha_1 p^{\ell-k} a+ \beta_1 b)+\alpha_2 p^{t_1-k}a$ or $(\alpha_1 p^{\ell-k} a+ \beta_1 b)+\beta_2 p^{t_1-\ell}b$, where $0< \alpha_2 < p$ or $0< \beta_2 < p$.
	 
	 If $s=r$, then our dependency has the form
	 $$
	 (\alpha_1 p^{\ell-k} a+ \beta_1 b)+p^{t_1-\ell}(\alpha_2 \cdot p^{\ell-k} a +\beta_2   b)
	 $$
	 and therefore there exists a dependency of the same form with $0\leqslant \alpha_2,\beta_2< p$. The pairs $(\alpha_1,\beta_1)$ and $(\alpha_2,\beta_2)$ cannot be proportional, since otherwise it is just the sum 
	 $$
	 (\alpha_1 p^{\ell-k} a+ \beta_1 b)\cdot (1+c p^{t_1-\ell}),
	 $$
	 which is proportional to the first dependency.
	 
	 Let us show the third step. Imagine that there exists one more dependence which is not generated by the previous two dependencies. According to the choice of previous two dependencies it is a linear combination of $a$ and $b$  of a $p$-height~$t$, where $t\geqslant t_2$, so it has a form 
	 $$
	 (\alpha_1 + \alpha p^s)p^{\ell-k}a + (\beta_1+\beta p^r)b=(\alpha_1 p^{\ell-k} a+ \beta_1 b)+ ( \alpha p^sp^{\ell-k}a + \beta p^r b).
	 $$
	 The first summand in the right part has a $p$-height~$t_1$, it is a dependence, so the second summand  has a $p$-height~$t_1$. Let us subtract and add to the second summand the linear combination $p^{t_1-\ell}(\alpha_2 \cdot p^{\ell-k} a +\beta_2   b)$, and we will get
	 $$
	 (\alpha_1 p^{\ell-k} a+ \beta_1 b + p^{t_1-\ell}(\alpha_2 \cdot p^{\ell-k} a +\beta_2   b))+ (\gamma p^{\ell-k}  a + \delta b),
	 $$
	 the left summand has a $p$-height~$t_2$, this dependency is not generated by the two previous dependencies, so $\gamma a + \delta b$ has a $p$-height~$t_2$ and a $p$-height of the sum is $t_3> t_2$. As previously, $\gamma$ and $\delta$ cannot be proportional to $\alpha_1+p^{t_1-\ell} \alpha_2$ and $\beta_1+p^{t_1-\ell} \beta_2$. So we can see that finally we have the variant (4) of the theorem.
	 
	 We saw in Example~\ref{ex-p-adic} that for every prime~$p$ all these dependencies can be realized in the group $J_p\oplus J_p \oplus \mathbb Q \oplus \mathbb Q$ togehter with all infinite $q$-heights of all elements for all prime $q\ne p$. Therefore the whole type with all dependencies of all $p$-heights for all~$p$ can be realized in the direct product of these groups.
\end{proof}

\medskip

Unfortunately we see from the last theorem that the set of possible different standard $2$-types has the cardinality continuum and it does not look that it is possible to find some simple invariants for two Abelian torsion free groups to be $2$-isotypic.

\medskip

In any case, we have described all standard $2$-types of elements in such groups. It is easy to understand that $n$-types are some generalization of $2$-types: 

(1) It is possible to assume that all  $a_1,\dots, a_n$ are independent.

(2) If $p$-heights of some $a_{i_1},\dots, a_{i_m}$ are infinite or there exist linear combinations of elements $a_1,\dots, a_n$ with infinite $p$-height, then we can assume that $a_1,\dots, a_m\in \mathbb Q \oplus \dots \oplus \mathbb Q$ and $a_{m+1},\dots, a_n \in J_p\oplus \dots  \oplus J_p$, and these $a_{m+1},\dots, a_n$ can be independent, have a finite number of different dependencies or infinite number of different dependencies similarly to the case of $2$-tuples.

In any case, classification of all torsion free Abelian groups up to isotypicity does not look possible, as well as of the class of all Abelian groups. In the next section we will classify up to isotypicity one very popular class of torsion free Abelian groups.

\section{Isotypicity of torsion free separable groups}

\begin{definition}
	An Abelian torsion free group $A$ is called \emph{fully decomposable}, if it is a direct sum of groups of the rank~$1$. 
	
	Free and divisible groups are trivial examples of fully decomposable groups.
\end{definition}

\begin{proposition}[\cite{Baer6}]
	Any two decompositions of a fully decomposable group into a direct sum of groups of the rank~$1$ are isomorphic.
\end{proposition}

Let $A=\bigoplus\limits_{i\in I} A_i$, where $A_i$ are rational groups (subgroups of~$\mathbb Q$). Then for any type~$\mathbf t$ the subgroups $A(\mathbf t)$ and $A^* (\mathbf t)$ coincide with direct sums of those groups~$A_i$, for which $\mathbf t (A_i)\geqslant \mathbf t$ and $\mathbf t (A_i) > \mathbf t$, respectively. Therefore the group
$$
A_{\mathbf t} = A(\mathbf t) / A^* (\mathbf t)
$$
is isomorphic to the direct sum of those groups $A_i$, types of which are precisely~$\mathbf t$. We obtain that the rank of~$A_{\mathbf t}$ is equal to the number of summands $A_i$ of the type~$\mathbf t$. So for fully decomposable groups $A$ the ranks $r(A_{\mathbf t})$, where $\mathbf t$ belongs to the set of all types, form complete and independent system of invariants.


\smallskip

For the following popular class of torsion free Abelian group classification up to isotypicity is absolutely natural.

\begin{definition}
	An Abelian torsion free group $A$ is called \emph{separable} (see~\cite{Baer6}), if any finite subset of elements from~$A$ is contained in some fully decomposable direct summand of~$A$.
\end{definition}
 
\begin{theorem}
	Two torsion free separable Abelian groups $A_1$ and $A_2$ are isotypically equivalent if and only if for any h-type~$\mathbf t$ and for any natural~$n$ if there exist $n$ independent elements of this type~$\mathbf t$ in~$A_1$, then there exist $n$ independent elements of the type~$\mathbf t$ in~$A_2$, and vise versa.
\end{theorem}

\begin{proof}
	To prove this theorem we need to describe standard $n$-types of elements in separable torsion free Abelian groups.

	Let us have some $n$-tuple $(a_1,\dots, a_n)$ of elements in such a group~$A$. Then they are contained in a fully decomposable direct summand $B\subseteq A$. We can assume that $B=\bigoplus\limits_{i=1}^N B_i$, where each $B_i$ has a rank~$1$ and is described by its h-type~$\mathbf t_i$.  Let us take instead of elements $a_1,\dots, a_n$ their projections on the summands $B_i$. Each summand is a group of the rank~$1$, i.\,e., a subgroup of~$\mathbb Q$. For any finite set of elements $b_{i,1},\dots ,b_{i,n}$ of such a group~$B_i$  we can find an element $c_i\in B_i$ such that all these $b_{i,1},\dots , b_{i,n}$  are expressed trough~$c_i$. Therefore now instead of the initial $n$-tuple $(a_1,\dots, a_n)$ we have an $N$-tuple $(c_1,\dots, c_N)$ such that
	
	(1) all $a_1,\dots, a_n$ are expressed trough $c_1,\dots, c_N$;
	
	(2) all $c_1,\dots, c_N$ belong to different direct summands $B_i$ of~$A$;
	
	(3) every $c_i$ and the corresponding summand $B_i$ are fully characterized by its characteristic $\chi_i$ or by its h-type~$\mathbf t_i$.
	
	This $N$-tuple can be considered as standard. Now we know that any standard $N$-tuple $(c_1,\dots, c_N)$ is fully described by a tuple of characteristics $(\mathbf t_1,\dots, \mathbf t_N)$. It evidently implies the assertion of this theorem.
\end{proof}

The following two classical results have a useful corollary about isotypical equivalence:

\begin{theorem}[Baer, \cite{Baer6}]
	Two torsion free Abelian groups of the rank~$1$ are isomorphic if and only if they have the same height types. Every height type is a type of some rational group.
\end{theorem} 

\begin{theorem}[Baer, \cite{Baer6}]
	Any countable separable torsion free Abelian group is fully decomposable.
\end{theorem} 

The following corollary follows directly from the three previous theorems:

\begin{corollary}
	Two countable separable torsion free Abelian groups are isotypically equivalent if and only if they are isomorphic.
\end{corollary}

\medskip

Unfortunately that question, \emph{if there exist two countable isotypically equivalent groups, that are not isomorphic}, is still open.

\bigskip
{\bf Acknowledgements.}
Our sincere thanks go to Alexey Miasnikov and Eugene Plotkin for very useful discussions regarding  this work and permanent attention to it.

\end{document}